\documentclass[11pt,leqno]{article}
\usepackage[margin=1in]{geometry} 
\usepackage{amssymb,amsfonts,amsmath,bbm,mathrsfs,stmaryrd,mathtools}
\usepackage{xcolor}
\usepackage{url}

\usepackage{extarrows}

\usepackage[shortlabels]{enumitem}
\usepackage{tensor}

\usepackage{xr}
\usepackage[T1]{fontenc}

\usepackage[colorlinks,
linkcolor=black!75!red,
citecolor=blue,
pdftitle={},
pdfproducer={pdfLaTeX},
pdfpagemode=None,
bookmarksopen=true,
bookmarksnumbered=true,
backref=page]{hyperref}

\usepackage{tikz}
\usetikzlibrary{arrows,calc,decorations.pathreplacing,decorations.markings,intersections,shapes.geometric,through,fit,shapes.symbols,positioning,decorations.pathmorphing}

\usepackage{braket}

\usepackage[amsmath,thmmarks,hyperref]{ntheorem}
\usepackage{cleveref}

\newcommand\nthalias[1]{\AddToHook{env/#1/begin}{\crefalias{lemma}{#1}}}

\nthalias{definition}
\nthalias{example}
\nthalias{examples}
\nthalias{remark}
\nthalias{remarks}
\nthalias{convention}
\nthalias{notation}
\nthalias{construction}
\nthalias{theoremN}
\nthalias{propositionN}
\nthalias{corollaryN}
\nthalias{lemma}
\nthalias{proposition}
\nthalias{corollary}
\nthalias{theorem}
\nthalias{conjecture}
\nthalias{question}
\nthalias{assumption}

\creflabelformat{enumi}{#2#1#3}

\crefname{section}{Section}{Sections}
\crefformat{section}{#2Section~#1#3} 
\Crefformat{section}{#2Section~#1#3} 

\crefname{subsection}{\S}{\S\S}
\AtBeginDocument{%
  \crefformat{subsection}{#2\S#1#3}%
  \Crefformat{subsection}{#2\S#1#3}% 
}

\crefname{subsubsection}{\S}{\S\S}
\AtBeginDocument{%
  \crefformat{subsubsection}{#2\S#1#3}%
  \Crefformat{subsubsection}{#2\S#1#3}% 
}

\theoremstyle{plain}

\newtheorem{lemma}{Lemma}[section]
\newtheorem{proposition}[lemma]{Proposition}
\newtheorem{corollary}[lemma]{Corollary}
\newtheorem{theorem}[lemma]{Theorem}

\theoremstyle{plain}
\theoremnumbering{Alph}
\newtheorem{theoremN}{Theorem}
\newtheorem{propositionN}[theoremN]{Proposition}

\theoremstyle{plain}
\theorembodyfont{\upshape}
\theoremsymbol{\ensuremath{\blacklozenge}}

\newtheorem{definition}[lemma]{Definition}
\newtheorem{example}[lemma]{Example}

\newtheorem{remark}[lemma]{Remark}
\newtheorem{remarks}[lemma]{Remarks}

\newtheorem{notation}[lemma]{Notation}

\crefname{definition}{definition}{definitions}
\crefformat{definition}{#2definition~#1#3} 
\Crefformat{definition}{#2Definition~#1#3} 

\crefname{ex}{example}{examples}
\crefformat{example}{#2example~#1#3} 
\Crefformat{example}{#2Example~#1#3} 

\crefname{exs}{example}{examples}
\crefformat{examples}{#2example~#1#3} 
\Crefformat{examples}{#2Example~#1#3} 

\crefname{remark}{remark}{remarks}
\crefformat{remark}{#2remark~#1#3} 
\Crefformat{remark}{#2Remark~#1#3} 

\crefname{remarks}{remark}{remarks}
\crefformat{remarks}{#2remark~#1#3} 
\Crefformat{remarks}{#2Remark~#1#3} 

\crefname{convention}{convention}{conventions}
\crefformat{convention}{#2convention~#1#3} 
\Crefformat{convention}{#2Convention~#1#3} 

\crefname{notation}{notation}{notations}
\crefformat{notation}{#2notation~#1#3} 
\Crefformat{notation}{#2Notation~#1#3} 

\crefname{table}{table}{tables}
\crefformat{table}{#2table~#1#3} 
\Crefformat{table}{#2Table~#1#3}

\crefname{lemma}{lemma}{lemmas}
\crefformat{lemma}{#2lemma~#1#3} 
\Crefformat{lemma}{#2Lemma~#1#3} 

\crefname{proposition}{proposition}{propositions}
\crefformat{proposition}{#2proposition~#1#3} 
\Crefformat{proposition}{#2Proposition~#1#3} 

\crefname{propositionN}{proposition}{propositions}
\crefformat{propositionN}{#2proposition~#1#3} 
\Crefformat{propositionN}{#2Proposition~#1#3} 

\crefname{corollary}{corollary}{corollaries}
\crefformat{corollary}{#2corollary~#1#3} 
\Crefformat{corollary}{#2Corollary~#1#3} 

\crefname{corollaryN}{corollary}{corollaries}
\crefformat{corollaryN}{#2corollary~#1#3} 
\Crefformat{corollaryN}{#2Corollary~#1#3} 

\crefname{theorem}{theorem}{theorems}
\crefformat{theorem}{#2theorem~#1#3} 
\Crefformat{theorem}{#2Theorem~#1#3} 

\crefname{theoremN}{theorem}{theorems}
\crefformat{theoremN}{#2theorem~#1#3} 
\Crefformat{theoremN}{#2Theorem~#1#3} 

\crefname{enumi}{}{}
\crefformat{enumi}{#2#1#3}
\Crefformat{enumi}{#2#1#3}

\crefname{assumption}{assumption}{Assumptions}
\crefformat{assumption}{#2assumption~#1#3} 
\Crefformat{assumption}{#2Assumption~#1#3} 

\crefname{construction}{construction}{Constructions}
\crefformat{construction}{#2construction~#1#3} 
\Crefformat{construction}{#2Construction~#1#3} 

\crefname{question}{question}{Questions}
\crefformat{question}{#2question~#1#3} 
\Crefformat{question}{#2Question~#1#3} 

\crefname{equation}{}{}
\crefformat{equation}{(#2#1#3)} 
\Crefformat{equation}{(#2#1#3)}

\numberwithin{equation}{section}
\renewcommand{\theequation}{\thesection-\arabic{equation}}

\theoremstyle{nonumberplain}
\theoremsymbol{\ensuremath{\blacksquare}}

\newtheorem{proof}{Proof}
\newcommand\pf[1]{\newtheorem{#1}{Proof of \Cref{#1}}}

\newcommand\bC{{\mathbb C}}

\newcommand\bE{{\mathbb E}}

\newcommand\bG{{\mathbb G}}

\newcommand\bK{{\mathbb K}}

\newcommand\bP{{\mathbb P}}

\newcommand\bR{{\mathbb R}}

\newcommand\bZ{{\mathbb Z}}

\newcommand\cC{{\mathcal C}}

\newcommand\cM{{\mathcal M}}

\newcommand\cO{{\mathcal O}}

\newcommand\cS{{\mathcal S}}

\newcommand\1{{\bf 1}}

\DeclareMathOperator{\id}{id}

\DeclareMathOperator{\End}{\mathrm{End}}

\DeclareMathOperator{\Spec}{\mathrm{Spec}}

\DeclareMathOperator{\im}{\mathrm{im}}

\DeclareMathOperator{\spn}{\mathrm{span}}

\DeclareMathOperator{\GL}{GL}
\DeclareMathOperator{\PSL}{PSL}
\DeclareMathOperator{\SL}{SL}
\DeclareMathOperator{\PSU}{PSU}

\DeclareMathOperator{\G}{G}

\DeclareMathOperator{\Res}{Res}

\newcommand\numberthis{\addtocounter{equation}{1}\tag{\theequation}}

\newcommand{\cat}[1]{\textsc{#1}}

\newcommand{\qedhere}{\mbox{}\hfill\ensuremath{\blacksquare}}

\newcommand{\xrightarrowdbl}[2][]{%
  \xrightarrow[#1]{#2}\mathrel{\mkern-14mu}\rightarrow
}

\title{Generically-constrained quantum isotropy}
\author{Alexandru Chirvasitu}

\begin{document}

\date{}

\newcommand{\Addresses}{{% additional braces for segregating \footnotesize
  \bigskip
  \footnotesize

  \textsc{Department of Mathematics, University at Buffalo}
  \par\nopagebreak
  \textsc{Buffalo, NY 14260-2900, USA}  
  \par\nopagebreak
  \textit{E-mail address}: \texttt{achirvas@buffalo.edu}

  % % \medskip
  % % 
  % % \textsc{Department of Mathematics, INSTITUTION}
  % % \par\nopagebreak
  % % \textsc{ADDRESS}
  % % \par\nopagebreak
  % % \textit{E-mail address}: \texttt{??}
  % % 

}}

\maketitle

\begin{abstract}
  Let $V$ be a finite-dimensional unitary representation of a compact quantum group $\mathrm{G}$ and denote by $\mathrm{G}_W$ the isotropy subgroup of a linear subspace $W\le V$ regarded as a point in the Grassmannian $\mathbb{G}(V)$. We show that the space of those $W\in \mathbb{G}(V)$ for which $\mathrm{G}_W$ acts trivially on $W$ (or $V$) is open in the Zariski topology of the Weil restriction $\mathrm{Res}_{\mathbb{C}/\mathbb{R}}\mathbb{G}(V)$. More generally, this holds for the space of $W$ for which (a) the $\mathrm{G}_W$-action factors through its abelianization, or (b) the summands of the $\mathrm{G}_W$-representation on $W$ (or $V$) are otherwise dimensionally constrained.

  The results generalize analogous classical generic rigidity statements useful in establishing the triviality of the classical automorphism groups of random quantum graphs in the matrix algebra $M_n$, and can be put to similar use in fully non-commutative versions of those results (quantum graphs, quantum groups). 
\end{abstract}

\noindent {\em Key words:
  Grassmannian;
  Hopf algebra;
  Tannaka reconstruction;
  Zariski topology;
  coalgebra;
  comodule;
  constructible;
  linear algebraic group

}

\vspace{.5cm}

\noindent{MSC 2020: 20G42; 16T15; 18M05; 46L67; 14L15; 18M25; 20G15; 14M15

  %20G42 Quantum groups (quantized function algebras) and their representations
  %16T15 Coalgebras and comodules; corings
  %18M05 Monoidal categories, symmetric monoidal categories
  %20G15 Linear algebraic groups over arbitrary fields
  %46L67 Quantum groups (operator algebraic aspects)
  %14L15 Group schemes
  %14M15 Grassmannians, Schubert varieties, flag manifolds
  %18M25 Tannakian categories

}

%\tableofcontents

%%%%%%%%%%%%%%%%%%%%%%%%%%%%%%%%
%%%%%%%%%%%%%%%%%%%%%%%%%%%%%%%%
\section*{Introduction}

The present note's results instantiate in a number of ways the familiar general principle that ``most'' objects of any given type exhibit ``little'' symmetry. One germane starting point is \cite{zbMATH07502493}, examining that phenomenon in the context of \emph{quantum graphs} \cite[Definition 4.3]{MR4706978} housed by the $n\times n$ matrix algebra $M_n$ as \emph{operator systems} \cite[p.9]{pls-bk} therein: self-adjoint complex subspaces $\cS\le M_n$ containing $1\in M_n$. \cite[Theorem 3.19]{zbMATH07502493} proposes one way to formalize the notion that such $\cS$ are mostly rigid, relying among other things on the auxiliary result of \cite[Lemma 3.12]{zbMATH07502493}:
\begin{itemize}[wide]
\item given a finite-dimensional (complex, say) representation of a compact group $\G$ on $V$;

\item and a subspace $W$ with
  \begin{equation}\label{eq:gww}
    \text{restriction }\G_W|_W
    =
    \left\{\id_W\right\}
    ,\quad
    \G_W:=\left\{g\in \G\ :\ gW\le W\right\},
  \end{equation}
\end{itemize}
condition \Cref{eq:gww} holds for all $W'$ in a neighborhood of $W$ in the \emph{Grassmannian} \cite[\S 5.1]{ms_nonl} $\bG(\dim W,V)$ (the statement is specialized to compact \emph{Lie} groups, but the proof is not). More is true: the set
\begin{equation*}
  \left\{W\in \bG(d,V)\ :\ \G_W|_{W} = \left\{\id_{W}\right\}\right\}
  \subseteq
  \bG(d,V)
\end{equation*}
is in fact open in the \emph{Zariski} (rather than just ordinary) topology. They are thus also ``large'' in any reasonable sense one might accord the term:
\begin{itemize}[wide]
\item having full measure in every open subset of $\bG(d,V)$ intersecting it;
  
\item or being \emph{residual} (i.e. \cite[\S 9, p.41]{oxt_meas-cat_2e_1980} the complement of a countable union of nowhere-dense sets).
\end{itemize}
Both of these follow if largeness is interpreted as algebraic geometers not uncommonly do (e.g. \cite[pp.1926, 1927, 1948]{rs_ess-dim-2}): containing a countable intersection of Zariski-open dense subsets. Either condition would, incidentally, also have sufficed for the applications of \cite[Lemma 3.12]{zbMATH07502493} in that paper.

Most of the results below are variations on this selfsame theme that the sets of $W\in \bG(d,V)$ satisfying \Cref{eq:gww} (or closely analogous conditions) are in some fashion substantial. The precise setting admits some flexibility, with the groups in question ranging from linear algebraic to quantum.

The perspective on quantum groups $\G$ as Hopf algebras $\cO(\G)$ is that of \cite[\S 13.1]{cp_qg} (and \cite[\S 2]{wor-cqg}, etc.), the correspondence being contravariant:
\begin{itemize}[wide]
\item a Hopf algebra $\cO(\G)$ is to be regarded as algebra of particularly well-behaved (\emph{representative}: \cite[Chapter VI, Definition V.1]{chev_lie-bk_1946}, \cite[\S 27.7]{hr-2}) functions on the otherwise non-existent quantum group $\G$;

\item and $\G$ is referred to as a \emph{compact quantum} group \cite[p.65, Definition]{tim} if $\cO(\G)$ is a \emph{CQG algebra} in the sense of \cite[Definition 2.2]{dk_cqg}: a complex \emph{Hopf $*$-algebra} \cite[\S 1, p.316]{dk_cqg} whose finite-dimensional comodules
  \begin{equation*}
    V \ni
    v
    \xmapsto{\quad}
    v_0\otimes v_1
    \in
    V\otimes \cO(\G)
  \end{equation*}
  (\emph{Sweedler notation} \cite[\S 1.2 and post Proposition 2.0.1]{swe}) are \emph{unitarizable}. The latter phrase means that $V$ admits an inner product $\Braket{-\mid-}$ invariant under the $\G$-action in the sense of \cite[(1.8)]{dk_cqg}:
  \begin{equation}\label{eq:untrz}
    \Braket{v_0\mid w}v_1^*
    =
    \Braket{v\mid w_0}Sw_1
    ,\quad
    \forall v,w\in V.
  \end{equation}  
\end{itemize}

One goal will be to quantize the already-cited \cite[Lemma 3.12]{zbMATH07502493} along with analogues: \cite[Corollary 3.9]{zbMATH07502493}, for instance, implies that whenever the (classical) automorphism group of a single operator system in $M_n$ is abelian, it is so over a non-empty Zariski-open subset of the Grassmannian. The phrasing of \Cref{thn:qg.triv.ab} employs quantum-group vocabulary; its proof (following that of the more general \Cref{th:cqg.triv.isotr.zopen}) effects a translation more appropriate to the language and notation introduced there. 

There is a caveat to what \emph{Zariski-open} means: as we are interested in complex representations $V$ of \emph{compact} groups, the Zariski topology will not quite be the usual one on the complex Grassmannian $\bG:=\bG(d,V)$; it is rather the stronger one on the \emph{Weil restriction} $\Res_{\bC/\bR}\bG$. Per \cite[\S 4.6]{poon_rat-pts} say, in turn referring back to \cite[\S 7.6]{blr_neron}, this is the variety obtained by regarding the complex points of the Grassmannian as the real points of a scheme over $\bR$ instead.

\begin{theoremN}\label{thn:qg.triv.ab}
  Let $V$ be a finite-dimensional representation of a compact quantum group $\G$, and denote by $\G_W\le \G$ the isotropy quantum subgroup of $W\in \bG(d,V)$. The subsets
  \begin{align*}
    &\left\{
      W\in \bG(d,V)
      \ :\
      \G_W\text{ acts trivially on $W$ or $V$}
      \right\}
      \subseteq
      \bG(d,\dim V)\\
    &\left\{
      W\in \bG(d,V)
      \ :\
      \text{the $\G_W$-action on $W$ or $V$ factors through the abelianization $(\G_W)_{ab}$}
      \right\}
  \end{align*}
  are open in the Zariski topology on $\Res_{\bC/\bR}\bG(d,\dim V)$, and hence dense in the standard topology if non-empty.  \qedhere
\end{theoremN}

By analogy to the classical situation, whereby a (locally) compact group is abelian precisely when its irreducible unitary representations are all 1-dimensional (noted in \cite[\S 1.7, pre Definition 1.72]{kt}, say), \emph{abelian} compact quantum groups $\G$ are those with cocommutative $\cO(\G)$. The \emph{abelianization} $\G_{ab}$ of a compact quantum group is thus the (object dual to the) largest cocommutative subcoalgebra of $\cO(\G)$ (automatically a CQG subalgebra). This is also the \emph{coreflection} \cite[post Definition 3.1.4]{brcx_hndbk-1} of $\cO(\G)$ in the category of cocommutative CQG algebras (cf. \cite[p.12, $2^{nd}$ diagram]{porst_formal-2}). 

A number of possible directions to pursue suggest themselves. On the one hand, \Cref{thn:qg.triv.ab} specializes to classical compact groups. On the other, the results for the latter can also be recovered from openness in the \emph{standard} (Hausdorff) topology (\cite[Lemma 3.12]{zbMATH07502493} again) coupled with the \emph{constructibility} of the relevant sets: recall \cite[AG, \S 1.3]{brl} that subsets of a topological space are
\begin{itemize}[wide]
\item \emph{locally closed} if open in their closure;

\item and \emph{constructible} if they are finite disjoint unions of locally closed sets. 
\end{itemize}
Openness in some sense reflects the compactness of $\G$ (see \Cref{ex:axb} and \Cref{res:cpct.needed}\Cref{item:res:cpct.needed:also.alg}) but constructibility holds in the broader context of linear-algebraic-group representations. The following simple remark is a particular instance of \Cref{pr:cnstrctbl}. 

\begin{propositionN}\label{prn:cnstrctbl}
  Let $\Bbbk$ be a field, $\G$ an affine $\Bbbk$-group and $V$ a $\G$-representation. The linear subspaces $W\le V$ whose isotropy subgroups $\G_W\le \G$ operate trivially on $W$ constitute a constructible subset of the Grassmannian of $V$.  \qedhere
\end{propositionN}

A side issue that is nevertheless ubiquitous throughout \cite{zbMATH07502493} is the fact that trivial actions on operator systems $\cS\le M_n$ give trivial actions on $M_n$ for the simple reason that ``most'' such operator systems generate $M_n$ as an algebra. The result is well-known and appears in many guises in the operator-algebra literature: the self-adjoint generating tuples $n$-tuples of a finite-dimensional $C^*$-algebra $A$ constitute a dense open subset of $A^n_{sa}$ (with the subscript denoting self-adjoint elements) whenever $n\ge 2$, and in fact the complement of a finite union of submanifolds (\cite[Lemma 7.2]{zbMATH07289424}, \cite[Lemma 4.9]{MR4620324}, etc.). 

We take the opportunity to record the algebraic analogue, which in turn recovers the Zariski density of the set of generating $(\ge 2)$-tuples for finite-dimensional $C^*$-algebras: \Cref{cor:gen.cast.zar.den}. 

\begin{theoremN}[\Cref{th:gen.zar.den.sep}]
  The set of generating $n$-tuples for a finite-dimensional algebra $A$ over a field $\Bbbk$ constitute a Zariski-open $\Bbbk$-defined set in $A^n$, dense if $n\ge 2$ and the algebra is separable.  \qedhere
\end{theoremN}

%%%%%%%%%%%%%%%%%%%%%%%%%%%%%%%%
\subsection*{Acknowledgments}

I am grateful for helpful comments, pointers and suggestions from M. Fulger, K. Goodearl, R. Kanda, S. P. Smith, P. So{\l}tan, V. Srinivas, D. \c{S}tefan, H. Thiel and M. Wasilewski.

% % %%%%%%%%%%%%%%%%%%%%%%%%%%%%%%%%
% % %%%%%%%%%%%%%%%%%%%%%%%%%%%%%%%%
% % \section{Preliminaries}\label{se:prel}
% %

%%%%%%%%%%%%%%%%%%%%%%%%%%%%%%%%
%%%%%%%%%%%%%%%%%%%%%%%%%%%%%%%%
\section{Generic trivial (quantum) actions}\label{se:gen.triv.qact}

We need some algebraic-geometry background, and refer the reader to \cite[Chapter AG]{brl} for most of it. Some of the specific notions employed below include
\begin{itemize}[wide]
\item the \emph{Zariski topology} \cite[AG \S 3.4]{brl} on the underlying topological space of a \emph{$\bK$-scheme} \cite[AG \S 3.4]{brl} for an algebraically closed field $\bK$;

\item \emph{$\Bbbk$-structures} and the \emph{$\Bbbk$-topology} \cite[AG \S 11.3]{brl} on the same objects, for subfields $\Bbbk\le \bK$, weaker than the aforementioned Zariski ($\bK$-)topology;

\item and downstream from that, subschemes \emph{defined over} \cite[AG \S 11.4]{brl} the subfield $\Bbbk\le \bK$.

  For issues of \emph{rationality} (i.e. being $\Bbbk$-closed or defined) see also \cite[\S 34.1]{hmph_lin-alg-gps_1981}. 
\end{itemize}
We refer the reader to \cite[\S I.2]{jntz}, \cite[\S II.2.2]{dg_gp-alg} or \cite[\S 5.4]{dmdts_lalg-gps} for basic background on $\Bbbk$-representations of affine $\Bbbk$-groups $\G$ (for a field $\Bbbk$). The Zariski openness of \Cref{eq:gww} will be a consequence of openness in the standard topology and the type of constructibility result \Cref{pr:cnstrctbl} and (the proof of) \Cref{pr:triv.act.zar.op} below provide.

% % For a family $\{W\}$ of subspaces of $V$ and a $\G$-representation on $V$ we write $\G_{\{W\}}$ for the subgroup of $\G$ leaving each $W\in \{W\}$ invariant. We occasionally drop braces, as in $\G_{W}:=\G_{\{W\}}$. 
% %

%\newpage

\begin{proposition}\label{pr:cnstrctbl}
  Let $\G$ be an affine $\Bbbk$-group for a field $\Bbbk$, and $V$ a finite-dimensional $\G$-representation over $\Bbbk$, and $\psi$ a $\Bbbk$-defined scheme morphism
  \begin{equation*}
    \bG(d,V)
    \xrightarrow{\quad\psi\quad}
    \bG(V)
    :=
    \bigcup_{d'}\bG(d',V).
  \end{equation*}
  The set
  \begin{equation}\label{eq:triv.gw.ws}
    \bG_{\psi}
    :=
    \left\{W\in \bG(d,V)\ :\ \G_W|_{\psi(W)} = \left\{\id_{\psi(W)}\right\}\right\}
    \subseteq
    \bG(d,V)
  \end{equation}
  is constructible.
\end{proposition}
\begin{proof}
  We may as well assume $\G$ of \emph{finite type} \cite[Definition post Proposition II.3.2]{hrt} over $\Bbbk$, by substituting for it its (automatically closed, $\Bbbk$-defined \cite[Corollary 1.4(a)]{brl}) image through $\G\to \GL(V)$. 
  
  We will exhibit \Cref{eq:triv.gw.ws} (or rather its complement) as a the image of a constructible set through morphisms of sufficiently well-behaved schemes (of finite type over $\Bbbk$). This will suffice by the celebrated \emph{Chevalley theorem} (e.g. \cite[Th\'eor\`eme IV.1.8.4]{ega41} or \cite[Exercise II.3.19]{hrt}).

  We write
  \begin{equation*}
    \begin{tikzpicture}[>=stealth,auto,baseline=(current  bounding  box.center)]
      \path[anchor=base] 
      (0,0) node (l) {$\bG(d,V)$}
      +(5,.5) node (u) {$\bE:=\left\{(W,v)\ :\ v\in \psi(W)\right\}\subseteq \bG(d,V)\times V$}
      +(10,0) node (r) {$V$}
      ;
      \draw[->>] (u) to[bend right=6] node[pos=.5,auto,swap] {$\scriptstyle \pi$} (l);
      \draw[->] (u) to[bend left=6] node[pos=.5,auto] {$\scriptstyle $} (r);
    \end{tikzpicture}
  \end{equation*}
  for the \emph{incidence correspondence} attached naturally to $\psi$; $\pi$ is also a \emph{$d'$-bundle} \cite[p.10]{3264} over $\bG(d,V)$ for the common dimension $d'$ of all $\psi(\bullet)$: the pullback through $\psi$ of the \emph{canonical subbundle} \cite[\S 3.2.3]{3264} on $\bG(d',W)$. We can then also speak of $\bE\oplus \bE$, the direct sum of vector bundles over $\bG(d,V)$. 

  To conclude, simply observe that the complement of \Cref{eq:triv.gw.ws} is the image through the first projection $\bE\oplus \bE\xrightarrowdbl{}\bG(d,V)$ of
  \begin{equation*}
    \left\{
      \left(W,v_0,v_1,g\right)
      \in
      \left(\bE\oplus \bE\right)\times \G
      \ :\
      gW=W
      \ \wedge\
      v_0\ne v_1
      \ \wedge\
      gv_0=v_1
    \right\}.
  \end{equation*}
  The latter set plainly being locally closed in $\left(\bE\oplus \bE\right)\times \G$, we are done. 
\end{proof}

We revisit \cite[Lemma 3.12]{zbMATH07502493}, leveraging its openness claim into the stronger Zariski openness, recalling that the appropriate interpretation of the phrase means taking Weil restrictions.

\begin{proposition}\label{pr:triv.act.zar.op}
  Let $\G\xrightarrow{\rho}\GL(V)$ be a finite-dimensional complex representation of a compact group. For any integer $0\le d\le \dim V$ the sets \Cref{eq:triv.gw.ws} are Zariski-open in $\Res_{\bC/\bR}\bG(d,V)$ and hence also dense whenever non-empty.
\end{proposition}
\begin{proof}
  There is no loss in assuming $\G$ Lie throughout, for $\rho$ in any case factors through a Lie quotient (e.g. the image $\rho(\G)$, a closed and hence \cite[Theorem 20.12]{lee2013introduction} Lie subgroup of $\GL$). Observe furthermore that given the \emph{irreducibility} \cite[Theorem 5.4]{ms_nonl} of $\bG:=\bG(d,V)$, density does indeed follow from Zariski openness and non-emptiness.

  Constructibility follows not quite from \Cref{pr:cnstrctbl}, but rather from a parallel argument. The possible $\G_W$ range over finitely many subgroups of $\G$ up to conjugation \cite[Proposition I.2.4]{audin}, so whenever $g\in \G_W$ fails to operate trivially on $W$ the spectrum of $g|_W$ must intersect some fixed finite subset $F\subset \bC\setminus\{1\}$. The complement of \Cref{eq:triv.gw.ws} is thus the image of 
  \begin{equation*}
    \left\{(W,g)\ :\ gW=W\ \wedge\ \det\prod_{\lambda\in F}(g-\lambda)|_W=0\right\}
    \subseteq
    \bG\times \G
  \end{equation*}
  through the first projection. That set is Zariski-closed in $\Res_{\bC/\bR}\bG\times \G$, having given $\G$ its real algebraic structure inherited from its embedding $\G<\G_{\bC}$ into the complex linear algebraic group attached to it \cite[VI, Definition VIII.3]{chev_lie-bk_1946}.

  % % OLD: BEFORE GENERALIZING \cite[Lemma 6.3.3]{zbMATH03382687} TO COUNTABLE UNIONS
  % % 
  % % by \cite[Lemma 6.3.3]{zbMATH03382687}, say: that source assumes all varieties complex, but the proof of the cited result goes through fine for real points of real varieties.
  % % 
  
  Openness in the standard topology is provided by \cite[Lemma 3.12]{zbMATH07502493}, whence Zariski openness (given constructibility) by \Cref{le:count.un.constr} below, slightly generalizing \cite[Lemma 6.3.3]{zbMATH03382687}.
\end{proof}

For completeness and good measure, we record the following extension of \cite[Lemma 6.3.3]{zbMATH03382687}: that source assumes varieties complex and the result in question applies to constructible sets rather than countable unions of such. 

\begin{lemma}\label{le:count.un.constr}
  A countable union $U\subseteq X$ of constructible subsets of a real algebraic variety $X\subseteq \bR\bP^n$ open in $X$ in the Hausdorff topology is Zariski-open. 
\end{lemma}
\begin{proof}
  We proceed by induction on the dimension of $X$, the base case $\dim V=0$ being trivial. \emph{Baire's theorem} \cite[Theorem 24.12 and Corollary 25.4(b)]{wil_top} applied to the countable union
  \begin{equation*}
    U = \bigcup_{n=0}^{\infty} C_n
    ,\quad
    C_n\text{ constructible}
  \end{equation*}
  implies that some Zariski closure $\overline{C_n}^{Z}$ must have non-empty interior. All varieties in sight are \emph{noetherian} in the sense of \cite[AG \S 1.2]{brl}, so $C_n$ contains some Zariski-open subset $W\subseteq X$ by \cite[AG \S 1.3, Proposition]{brl}. This argument goes through so long as the union of the $W$ thus obtained is not dense in $X$.
  
  On the other hand, by noetherianness again, the union $\bigcup_W W$ will eventually stabilize. It follows that $U$ does contain a Zariski-open dense subset $U_0\subseteq X$, and induction then takes over with $X\setminus U_0$ in place of $X$.
\end{proof}

The openness of the set of spaces $W$ satisfying \Cref{eq:gww} is a manifestation of sorts of the compactness of $\G$. Per \Cref{ex:axb}, that set need not be either open or in any sense ``large'' in general.

\begin{example}\label{ex:axb}
  Consider the group $\G:=\left(\bR,+\right)\rtimes \left(\bR_{>0},\cdot\right)$ (the second semidirect factor acting on the first by scaling) operating on $\bC^2$ via the realization
  \begin{equation*}
    \begin{pmatrix}
      1&b\\
      0&a
    \end{pmatrix}
    ,\quad
    a\in \bR_{>0}
    ,\quad
    b\in \bR
  \end{equation*}
  (the familiar \emph{$ax+b$ group} \cite[5. post Proposition 2.21]{folland}). The isotropy group of the line $\ell$ operates on it non-trivially with the single exception of $\ell_0:=\bC
  \begin{pmatrix}
    1\\0
  \end{pmatrix}
  $, so the set of $W$ (here $\ell$) satisfying \Cref{eq:gww} is a singleton in $\bC\bP^1=\bG(1,\bC^2)$. Indeed, this is so for $\ell_1:=\bC
  \begin{pmatrix}
    0\\1
  \end{pmatrix}
  $ because its isotropy group is the second factor $\bR_{>0}$, and all lines $\ell\ne \ell_0$ constitute a single orbit under the action of $\G$ on $\bC\bP^1$. 
\end{example}

\begin{remarks}\label{res:cpct.needed}
  \begin{enumerate}[(1),wide]
  \item\label{item:res:cpct.needed:also.alg} \Cref{ex:axb} adapts easily to an algebraic counterpart thereof, with the Zariski topology in place of the usual Euclidean one and the linear algebraic group $\G_a\rtimes \G_m$ doing the acting ($\G_a$ and $\G_m$ denoting, as usual \cite[Examples I.1.6(1) and (3)]{brl}, the additive and multiplicative 1-dimensional linear algebraic groups).

  \item\label{item:res:cpct.needed:weil} Note also that in \Cref{pr:triv.act.zar.op} requiring openness only in the stronger topology underlying the Weil restriction to $\bR$ is not gratuitous: it is \emph{not} the case, in general, that \Cref{eq:triv.gw.ws} are Zariski-open in the more restrictive sense in that result's context. \Cref{le:psl2vn} below (a gloss on \cite[Example 12.1]{zbMATH03382687}) demonstrates this. 
  \end{enumerate}  
\end{remarks}

\Cref{le:psl2vn} below refers to the $(2n+1)$-dimensional irreducible representation $V_{2n}$ of $\PSL(2)$, i.e. \cite[\S 11.1, p.150]{fh_rep-th} the $n^{th}$ symmetric power $S^n \bC^2$ of the usual $\SL(2)$-representation on $\bC^2$. $V_{2n}$ can be regarded \cite[\S 11.3]{fh_rep-th} as the space of homogeneous degree-$2n$ polynomials in $x,y$.

\begin{lemma}\label{le:psl2vn}
  Let $n\in \bZ_{\ge 2}$, $\G:=\PSU(2)<\PSL(2)$, and set $\bP:=\bP V_{2n}$.

  The collection of lines $\ell\le V_{2n}$ with $\G_{\ell}|_{\ell}=\{1\}$ is Zariski-open in $\Res_{\bC/\bR}\bP$ but not in $\bP$. 
\end{lemma}
\begin{proof}
  The positive claim is a consequence of \Cref{pr:triv.act.zar.op}. For its negative counterpart, consider the set
  \begin{equation*}    
    \left\{
      \ell_{p,q}:=\bC p^{2n-1}q
      \ :\ 
      p\ne q\text{ linear forms in $x,y$}
    \right\}.
  \end{equation*}
  It can be identified with the space of idempotent rank-1 operators on $\bC^2$ via
  \begin{equation*}
    \bC \ell_{p,q}
    \xmapsto{\quad}
    P_{p,q}
    \ \text{with}\ 
    \im P_{p,q}=\bC p\ \text{and}\ \ker P_{p,q}=\bC q,
  \end{equation*}
  whereupon its intersection with $\{\ell\ :\ \G_{\ell}|_{\ell}=\{1\}\}$ consists precisely of the self-adjoint idempotents with respect to the $\G$-invariant inner product on $\bC^2$. The conclusion follows. 
\end{proof}

% % OLD: REPLACED BY \Cref{le:psl2vn}
% % 
% % \begin{example}\label{ex:psl2v4}
% %   Consider the $5$-dimensional irreducible representation of $\PSL(2)$, i.e. \cite[\S 11.1, p.150]{fh_rep-th} the $4^{th}$ symmetric power $S^4 \bC^2$ of the usual $\SL(2)$-representation on $\bC^2$. $S^4 V$ can be regarded \cite[\S 11.3]{fh_rep-th} as the space of homogeneous degree-4 polynomials in $x,y$.
% % 
% %   The representation restricts to one of the maximal compact subgroup $\G:=\PSU(2)\subset \PSL(2)$, and I claim that the collection of lines $\ell\le S^4\bC^2$ with
% %   \begin{equation*}
% %     \G_{\ell}|_{\ell}=\{1\}
% %   \end{equation*}
% %   is not open in the usual Zariski topology on the relevant projective space $\bP:=\bP S^4\bC^2$. Indeed, that collection's intersection with the locally closed subspace
% %   \begin{equation*}
% %     \left\{\ell_{p,q}:=\bC p^3q\ :\ p\ne q\text{ linear forms in $x,y$}\right\}
% %     \subset \bP
% %   \end{equation*}
% %   consists of precisely those $\ell_{p,q}$ for which $p,q\in \bC^2$ are not orthogonal with respect to the inner product left invariant by the $\SU(2)$-action. Orthogonality is a \emph{real}-algebraic condition on complex projective space, not a \emph{complex}-algebraic one. 
% % \end{example}

%\newpage

We will be concerned with quantizing \Cref{pr:triv.act.zar.op}, i.e. adapting it to a quantum-action setting. We will take for granted much of the basics on coalgebras, bialgebras and Hopf algebras accessible in any number of good sources, such as \cite{abe,dnr,mont,rad,swe} and many others. More specific citations to those (and other sources) are scattered throughout.

\begin{notation}\label{not:cats}
  \begin{enumerate}[(1),wide]
  \item\label{item:not:cats:coalg} $\cat{Coalg}$, $\cat{Bialg}$ and $\cat{HAlg}$ denote the categories of coalgebras, bialgebras and Hopf algebras over a field $\Bbbk$ understood from context.

  \item\label{item:not:cats:cqg} Similarly, $\cat{HAlg}^*$ and $\cat{CQGAlg}$ denote the categories of complex Hopf $*$-algebras and CQG algebras respectively.

  \item\label{item:not:cats:comod} $\cM^C$ and $\cM^C_f$ denote the categories of right comodules and finite-dimensional comodules over a coalgebra respectively, and similar conventions hold for modules with subscripts in place of superscripts (e.g. $\tensor*[_A]{\cM}{}$ means left $A$-modules). For field $\Bbbk$ we also write $\cat{Vect}=\cat{Vect}_{\Bbbk}$ for $\tensor*[_{\Bbbk}]{\cM}{}$ for emphasis. 
  \end{enumerate}  
\end{notation}
We remind the reader that we employ Sweedler notation
\begin{equation*}
  C\ni
  x
  \xmapsto{\quad\Delta\quad}
  x_1\otimes x_2
  \in C\otimes C
  \quad\text{and}\quad
  V\ni
  v
  \xmapsto{\quad\rho\quad}
  v_0\otimes v_1
  \in V\otimes C
\end{equation*}
for coalgebra comultiplications and coactions. For CQG algebras $\cO(\G)$ we will typically assume finite-dimensional comodules $V\in \cM_f^{\cO(\G)}$ as \emph{unitarized}, i.e. equipped with inner products $\Braket{-\mid-}$ satisfying \Cref{eq:untrz}.

% % OLD: WRONG
% % 
% % We begin with a simple remark, spelled out for future reference.
% % 
% % \begin{proposition}\label{pr:cont.loc}
% %   Let $\Bbbk$ be a field and
% %   \begin{equation*}
% %     Y
% %     \xrightarrow{\quad S\quad}
% %     X
% %     \xleftarrow{\quad T\quad}
% %     Z
% %   \end{equation*}
% %   morphisms in $\tensor*[_{\Bbbk}]{\cat{Vect}}{_f}$. The subset
% %   \begin{equation*}
% %     \left\{(Y',Z')\in \bG(Y)\times \bG(Z)\ :\ SY'\ge TZ'\right\}
% %     \subseteq
% %     \bG(Y)\times \bG(Z)
% %   \end{equation*}
% %   consists of the $\Bbbk$-points of a $\Bbbk$-defined closed subscheme.
% % \end{proposition}
% % \begin{proof}
% % 
% %   \mgnt{wrong}
% % \end{proof}

Recall \cite[Lemma 2.1.9]{schau_tann} that a functor $\cC\xrightarrow{\omega}\cat{Vect}_f$ affords a \emph{coendomorphism coalgebra} $\cat{coend}(\omega)\in\cat{Vect}$ with
\begin{equation*}
  \cat{Vect}(C,-)
  \cong
  \cat{Nat}\left(\omega,\omega\otimes -\right)
  \quad
  \text{as functors}
  \quad
  \cat{Vect}\xrightarrow{\quad}\cat{Set},
\end{equation*}
$\cat{Nat}$ denoting natural transformations. $\cat{coend}(\omega)$ is (as the name suggests) automatically a coalgebra, and its universality property \cite[Theorem 2.1.12]{schau_tann} consists in a universal factorization
\begin{equation*}
  \begin{tikzpicture}[>=stealth,auto,baseline=(current  bounding  box.center)]
    \path[anchor=base] 
    (0,0) node (l) {$\cC$}
    +(2,-.5) node (d) {$\cat{Vect}_f$}
    +(5,0) node (r) {$\cM_f^{\cat{coend}(\omega)}$}
    ;
    \draw[->] (l) to[bend left=6] node[pos=.5,auto] {$\scriptstyle $} (r);
    \draw[->] (l) to[bend right=6] node[pos=.5,auto,swap] {$\scriptstyle \omega$} (d);
    \draw[->] (r) to[bend left=6] node[pos=.5,auto] {$\scriptstyle \cat{forget}$} (d);
  \end{tikzpicture}
\end{equation*}
In other words, $C$ equips the objects $\omega(c)$ with comodule structures compatible with the maps $\omega(f)$, and every other such structure afforded by a coalgebra $D$ results by scalar corestriction along a unique coalgebra morphism $\cat{coend}(\omega)\to D$. 

We will also regard collections of vector spaces and morphisms as functors with appropriate domains; examples include single vector spaces $V$ as functors $\{\bullet\}\to\cat{Vect}$ (defined on the free category on one object) and morphisms $W\to V$ as functors $\left\{\bullet\to \bullet\right\}\to \cat{Vect}$. In that setup, one can speak of
\begin{equation*}
  \cat{coend}(V)\cong V^*\otimes V
\end{equation*}
and the (vector-space) \emph{pushout} \cite[Definition 11.30]{ahs}
\begin{equation}\label{eq:coend.wv}
  \begin{tikzpicture}[>=stealth,auto,baseline=(current  bounding  box.center)]
    \path[anchor=base] 
    (0,0) node (l) {$V^*\otimes W$}
    +(4,.5) node (u) {$\cat{coend}(V)\cong V^*\otimes V$}
    +(4,-.5) node (d) {$\cat{coend}(W)\cong W^*\otimes W$}
    +(8,0) node (r) {$\cat{coend}(W\to V)$}
    ;
    \draw[->] (l) to[bend left=6] node[pos=.5,auto] {$\scriptstyle $} (u);
    \draw[->] (u) to[bend left=6] node[pos=.5,auto] {$\scriptstyle $} (r);
    \draw[->] (l) to[bend right=6] node[pos=.5,auto,swap] {$\scriptstyle $} (d);
    \draw[->] (d) to[bend right=6] node[pos=.5,auto,swap] {$\scriptstyle $} (r);
  \end{tikzpicture}
\end{equation}
(see \cite[Definition 2.1.8]{schau_tann} for a justification of this pushout description).

\begin{remark}\label{re:coalg.coeff}
  For a coalgebra $C$ and $V\in \cM^C_f$ the image of the canonical map $\cat{coend}(V)\to C$ resulting from the comodule structure $V\xrightarrow{\rho}V\otimes C$ on $V$ is precisely the \emph{($V$-)coefficient coalgebra} $C_V$ defined by
  \begin{equation*}    
    \begin{tikzpicture}[>=stealth,auto,baseline=(current  bounding  box.center)]
      \path[anchor=base] 
      (0,0) node (l) {$V^*\otimes V$}
      +(3,.5) node (u) {$V^*\otimes V\otimes C$}
      +(3,-.5) node (d) {$C_V$}
      +(6,0) node (r) {$C$}
      ;
      \draw[->] (l) to[bend left=6] node[pos=.5,auto] {$\scriptstyle \id_{V^*}\otimes\rho$} (u);
      \draw[->] (u) to[bend left=6] node[pos=.5,auto] {$\scriptstyle \mathrm{ev}\otimes\id_C$} (r);
      \draw[->>] (l) to[bend right=6] node[pos=.5,auto,swap] {$\scriptstyle $} (d);
      \draw[right hook->] (d) to[bend right=6] node[pos=.5,auto,swap] {$\scriptstyle $} (r);
    \end{tikzpicture}
  \end{equation*}
  (what \cite[Proposition 2.5.3(iv)]{dnr} denotes by $A$, with that source's $M$ playing the role of our $V$).
\end{remark}

Consider, now, a coalgebra $C$, a comodule $V\in \cM^C_f$, and a linear map  $W\to V$ (not a morphism of comodules!). The data yields a pushout square
\begin{equation*}
  \begin{tikzpicture}[>=stealth,auto,baseline=(current  bounding  box.center)]
    \path[anchor=base] 
    (0,0) node (l) {$\cat{coend}(V)$}
    +(4,.5) node (u) {$C$}
    +(4,-.5) node (d) {$\cat{coend}(W\to V)$}
    +(8,0) node (r) {$C_{W\to V}$}
    ;
    \draw[->] (l) to[bend left=6] node[pos=.5,auto] {$\scriptstyle $} (u);
    \draw[->] (u) to[bend left=6] node[pos=.5,auto] {$\scriptstyle $} (r);
    \draw[->] (l) to[bend right=6] node[pos=.5,auto,swap] {$\scriptstyle $} (d);
    \draw[->] (d) to[bend right=6] node[pos=.5,auto,swap] {$\scriptstyle $} (r);
  \end{tikzpicture}
\end{equation*}
in the category $\cat{Coalg}$ of $\Bbbk$-coalgebras. Equivalently, the forgetful functor $\cat{Coalg}\to \cat{Vect}$ being left adjoint \cite[Theorem 6.4.1]{swe}, this is also a pushout in $\cat{Vect}$. 

$C_{W \to V}$, in other words, is the universal recipient of a coalgebra morphism $C\to C_{W\to V}$ for which $W\to V$ is a comodule morphism when $V$ is equipped with its original $C$- (hence also $C_{W\le V}$-)comodule structure. When $W\to V$ is an inclusion, the right-hand top and bottom arrows in \Cref{eq:coend.wv} are surjective and injective respectively, so that $C\xrightarrowdbl{}C_{W\le V}$ is a quotient coalgebra: the largest quotient coalgebra of $C$, in fact, leaving invariant the linear subspace $W\le V$ of the $C$-comodule $V$. The discussion extends also to families $\{W\to V\}$ of morphisms, and we adapt the notation accordingly in the most guessable fashion:
\begin{equation*}
  C_{W,W'\to V}
  :=
  \left(C_{W\to V}\right)_{W'\to V}
  \cong
  \left(C_{W'\to V}\right)_{W\to V},
\end{equation*}
similarly with `$\le$' in place of `$\to$', etc.

These considerations apply to bialgebras, Hopf algebras, Hopf $*$-algebras or CQG algebras as well (as does the ancillary Tannaka-reconstruction machinery; bialgebras and Hopf algebras are treated in \cite[\S\S 2.3 and 2.4]{schau_tann} respectively): for an object $H$ of one of the categories $\cC$ in \Cref{not:cats}\Cref{item:not:cats:coalg} or \Cref{item:not:cats:cqg} and 
\begin{equation*}
  V\in \cM^H_f
  ,\quad
  W\xrightarrow{\quad}V
  \text{ in \cat{Vect}}
\end{equation*}
there is a $\cC$-morphism $H\to H_{W\to V}$ \emph{initial} \cite[Definition 7.1]{ahs} among
\begin{equation*}
  H\xrightarrow{\quad}H'\text{ in }\cC
  \quad:\quad
  W\xrightarrow{\quad} V
  \text{ is a morphism in }\cM^{H'}.
\end{equation*}

The primary interest is in the Zariski openness of the set
\begin{equation*}
  \left\{
    W\in \bG(d,V)
    \ :\
    H_{W\le V}\text{ coacts trivially on }W
  \right\}
  \subseteq
  \bG(d,\dim V),
\end{equation*}
but the relevant discussion lends itself to fairly natural generalizations. The triviality in question amounts to requiring that the only simple $H_{W\le V}$-comodule supported by (i.e. appearing as a summand in) $W$ be the trivial one. This is the most restrictive of a hierarchy of conditions one might impose on $H_{W\le V}$ instead.

\begin{definition}\label{def:bd.reps}
  For $\ell\in \bZ_{\ge 0}$ a coalgebra $H$ is \emph{$\ell$-constrained on $V\in \cM^H_f$} if the dimensions of the non-trivial irreducible $H$-comodules appearing as subquotients of $V$ are all $\le \ell$.

  Note that for $\ell=0$ this specializes to requiring that all irreducible subquotients be trivial. In particular, if $V$ is semisimple as an object of $\cM^H$, the coaction must be trivial. 
\end{definition}

\begin{example}\label{ex:qab}
  For a CQG algebra $H$, consider the condition 
\begin{equation}\label{eq:qab}
  \left\{
    W\in \bG(d,V)
    \ :\
    H_{W\le V}\text{ coacts on $W$ via an abelian quantum group}
  \right\}
\end{equation}
instead. Recalling from the Introduction that abelian quantum groups are those dual to cocommutative Hopf algebras, the condition is simply that $H_{W\le V}$ be 1-constrained on $W$ in the sense of \Cref{def:bd.reps}. Or: the coaction of $H_{W\le V}$ on $W$ factors through the largest cocommutative subcoalgebra $\left(H_{W\le V}\right)_{coc}\le H_{W\le V}$. 
\end{example}

We can now state one quantum parallel to \Cref{pr:cnstrctbl} and \Cref{pr:triv.act.zar.op}, serving also as a non-commutative version of \cite[Lemma 3.12]{zbMATH07502493}.

\begin{theorem}\label{th:cqg.triv.isotr.zopen}
  Let $H=\cO(\G)\in \cat{CQGAlg}$, $V\in \cM^H_f$ a finite-dimensional comodule, $F'\subseteq F$ an inclusion of finite sets and
  \begin{equation*}
    \mathbf{d}=(d_i)_{i\in F}
    ,\quad
    \mathbf{\ell}=(\ell_j\in [0,d_j])_{j\in F'}
  \end{equation*}
  tuples of non-negative integers. 

  The set
  \begin{equation*}
    \tensor*[_{\le \mathbf{\ell}}]{\bG}{_{F,F'}}
    :=
    \left\{
      \left(W_i\right)_{i\in F}\in \bG(\mathbf{d},V)
      \ :\
      H_{\left(W_i\le V\right)_i}\text{ is $\ell_{j}$-constrained on every $W_j$}
    \right\}
    \subseteq
    \bG(d,V)
  \end{equation*}
  is open in the Zariski topology on $\Res_{\bC/\bR}\bG(\mathbf{d},V)$
  for
  \begin{equation*}
    \bG(\mathbf{d},V)
    :=
    \prod_{i\in F}\bG(d_i,V),
  \end{equation*}
  and hence dense in the standard topology if non-empty.
\end{theorem}

% % OLD: SINGLE W AND W'
% % 
% % \begin{theorem}\label{th:cqg.triv.isotr.zopen}
% %   Let $H=\cO(\G)\in \cat{CQGAlg}$, $V\in \cM^H_f$ a finite-dimensional comodule, and an integer $d\in [0,\dim V]$.
% % 
% %   For $\ell,d\in \bZ_{\ge 0}$ and a morphism
% %   \begin{equation*}
% %     \Res_{\bC/\bR}
% %     \bG(d,V)
% %     \xrightarrow{\quad\psi\quad}
% %     \Res_{\bC/\bV}\bG(V)
% %     ,\quad
% %     \bG(V)
% %     :=
% %     \bigcup_{d'}\bG(d',V)
% %   \end{equation*}
% %   the set
% %   \begin{equation*}
% %     \tensor*[_{\le \ell}]{\bG}{_{\psi}}
% %     :=
% %     \left\{
% %       W\in \bG(d,V)
% %       \ :\
% %       H_{W\le V}\text{ is $\ell$-constrained on $\psi(W)$}
% %     \right\}
% %     \subseteq
% %     \bG(d,V)
% %   \end{equation*}
% %   is open in the Zariski topology on $\Res_{\bC/\bR}\bG(d,\dim V)$, and hence dense in the standard topology if non-empty.
% % \end{theorem}

An immediate consequence:

\begin{corollary}\label{cor:wv}
  Let $H\in \cat{CQGAlg}$, $V\in \cM^H_f$ a finite-dimensional comodule, and an integer $d\in [0,\dim V]$.

  All of the following subsets of $\bG(d,V)$ are open in the Zariski topology on $\Res_{\bC/\bR}\bG(d,\dim V)$, and hence dense in the standard topology if non-empty.
  \begin{align*}
    &\left\{
      W\in \bG(d,V)
      \ :\
      H_{W\le V}\text{ is $\ell$-constrained on $W$}
      \right\}\quad\left(\ell\in \bZ_{\ge 0}\right)\numberthis\label{eq:wgdv.ab.w}\\
    &\left\{
      W\in \bG(d,V)
      \ :\
      H_{W\le V}\text{ is $\ell$-constrained on $V$}
      \right\}.\numberthis\label{eq:wgdv.ab.v}\\
  \end{align*}  
\end{corollary}
\begin{proof}
  For the first statement take the inclusion $F'\subseteq F$ of \Cref{th:cqg.triv.isotr.zopen} to be the identity on a singleton; for the second, set
  \begin{equation*}
    \left(F'\subseteq F\right)
    :=
    \left(\{1\}\subset \{0,1\}\right)
  \end{equation*}
  with $d_1:=\dim V$, forcing $W_1=V$. 
\end{proof}

Note that for \emph{embeddings} $W\le V$ the canonical maps $H\to H_{W\le V}$ are quotients (in the appropriate category), and the above discussion makes plain the following remark. We follow \cite[p.16]{swe} (say) in writing
\begin{equation*}
  W^{\perp}
  :=
  \left\{f\in V^*\ :\ f|_W\equiv 0\right\}
  \quad\text{for}\quad
  W\le V\text{ in }\cat{Vect}. 
\end{equation*}

\begin{lemma}\label{le:desc.ker}
  For an object $H$ of a category $\cC$ among those of \Cref{not:cats}\Cref{item:not:cats:coalg} or \Cref{item:not:cats:cqg}, a comodule $V\in \cM^H_f$ and a linear subspace $W\le V$ the kernel of $H\xrightarrowdbl{}H_{W\le V}$ is
  \begin{itemize}[wide]
  \item the image $\im W^{\perp}\otimes W$ of $W^{\perp}\otimes W\le V^*\otimes W$ through the map
    \begin{equation}\label{eq:vw2h}
      V^*\otimes W
      \lhook\joinrel\xrightarrow{\quad\id\otimes (W\le V)\quad}
      V^*\otimes V
      \xrightarrow{\quad\text{induced by $V$'s $H$-comodule structure}\quad}
      H.
    \end{equation}
    
  \item the ideal generated by $\im W^{\perp}\otimes W$ if $\cC=\cat{Bialg}$;
    
  \item the Hopf ideal generated by $\im W^{\perp}\otimes W$ if $\cC=\cat{HAlg}$;
   
  \item and the Hopf $*$-ideal generated by $\im W^{\perp}\otimes W$ if $\cC\in \left\{\cat{HAlg}^*,\ \cat{CQGAlg}\right\}$.  \qedhere
  \end{itemize}
\end{lemma}

\pf{th:cqg.triv.isotr.zopen}
\begin{th:cqg.triv.isotr.zopen}
  Density once more follows from openness and the irreducibility of the Grassmannian. We regard all $H$-comodules as unitarized, and write $P_\bullet$ for the orthogonal projection onto a subspace $\bullet\le V\in \cM_f^H$ with respect to the $\G$-invariant inner product $\Braket{-\mid-}$ on $V$. 

  We focus on the simple case of singletons $F'=F$, so that $H_{(W_i\le V)}$ is simply $H_{W\le V}$; the added complications in passing to the general case are only notational, and the simplification will help avoid those. 
  
  \begin{enumerate}[(I),wide]
  \item \textbf{: The case $\ell=0$.} The Tannaka-Krein-reconstruction gadgetry of \cite[Theorems 1.2 and 1.3]{woronowicz88} allows for some convenient recasting of the condition that $W\in\bG:= \bG(d,V)$ belong to $\tensor*[_{\le 0}]{\bG}{_{(F,F')}}$. Consider finite sets $F=\left\{\alpha\right\}$ of (isomorphism classes of) simple $H$-comodules $V_{\alpha}$, and write
    \begin{equation*}
      A_{F,n}:=\cM^{H}\left(
        T^{\le n}\bigoplus_{\alpha\in F}\alpha,\ T^{\le n}\bigoplus_{\alpha\in F}\alpha
      \right)
      ,\quad
      T^{\le n}\bullet:=\bC\oplus \bullet\oplus\cdots\oplus \bullet^{\otimes n}
    \end{equation*}
    (the finite-dimensional $H$-comodule-endomorphism $C^*$-algebra). We then have
    \begin{equation}\label{eq:cont.endw}
      W\in \tensor*[_{\le 0}]{\bG}{_{(F,F')}}
      \xLeftrightarrow{\quad}
      \bigcup_{\substack{\text{finite }F\\n}}
      \Braket{P_W,\ A_{F,n}}\ge \End(W)
    \end{equation}
    $\Braket{\bullet}$ denoting the $*$-algebra generated by the arguments of the angled brackets, with generation understood to include tensoring operators. The latter condition can be rephrased as the requirement that
    \begin{equation*}
      P_{W}\cdot p_i\left(P_{W;m,k},\  x_j\right)\cdot P_{W}
      ,\quad
      1\le i\le d^2
    \end{equation*}
    be linearly independent for some choice of $x_j\in A_{F,n}$ and non-commutative $*$-polynomials $p_i$, where $P_{W;m,k}$ denotes the endomorphism
    \begin{equation*}
      \id_V^{\otimes(k-1)}\otimes P_W\otimes \id_V^{\otimes(m-k)}
      \in
      \End(V^{\otimes m})
    \end{equation*}
    (recall that $d=\dim W$, i.e. $d^2=\dim \End(W)$). In summary:
    \begin{equation*}
      \tensor*[_{\le 0}]{\bG}{_{(F,F')}}
      =
      \bigcup_{\substack{\text{finite }F\\x_j\in A_{F,n}\\\text{polynomials }p_i}}
      \left\{W\in \bG\ :\ P_{W}\cdot p_i\left(P_{W;m,k},\  x_j\right)\cdot P_{W}\text{ linearly independent}\right\}.
    \end{equation*}
    Each individual set on the right-hand side is plainly Zariski-open in $\Res_{\bC/\bR}\bG$, hence the conclusion.
    
  \item \textbf{: The case $\ell\ge 1$.} In the spirit of \Cref{eq:cont.endw}, the condition can be phrased as the requirement that some $\Braket{P_W,\ A_{F,n}}$ contain a family of projections $P_i\le P_{W}$ with
    \begin{equation*}
      \mathrm{rank}~P_i\le \ell
      ,\quad
      \sum_i P_i=1=\id_{W}.
    \end{equation*}
    This is also equivalent to the requirement that some normal operator
    \begin{equation*}
      P_{W}\cdot p\left(P_{W;m,k},\  x_j\right)\cdot P_{W}
      ,\quad
      x_j\in A_{F,n}
    \end{equation*}
    have all eigenvalue multiplicities $\le \ell$: again a Zariski-open condition.  \qedhere
  \end{enumerate}
\end{th:cqg.triv.isotr.zopen}

In particular, the triviality and abelianness of the quantum groups attached to $H_{W\le V}$ is an open condition; cf. \cite[Corollary 3.9]{zbMATH07502493}.

\pf{thn:qg.triv.ab}
\begin{thn:qg.triv.ab}
  After minor rephrasing, the statement claims that for a CQG algebra $H$ and $V\in \cM^H_f$ the subsets
  \begin{align*}
    &\left\{
      W\in \bG(d,V)
      \ :\
      H_{W\le V}\text{ coacts trivially on $W$ or $V$}
      \right\}
      \subseteq
      \bG(d,\dim V)\\
    &\left\{
      W\in \bG(d,V)
      \ :\
      H_{W\le V}\text{ coacts on $W$ or $V$ via an abelian quantum group}
      \right\}\\
  \end{align*}
  are open in the Zariski topology on $\Res_{\bC/\bR}\bG(d,\dim V)$, and hence dense in the standard topology if non-empty. Indeed: specialize \Cref{cor:wv} to $\ell=0$ or $1$.
\end{thn:qg.triv.ab}

\section{Generating operator algebras}\label{se:osgen}

Recall \cite[Lemma 7.2]{zbMATH07289424} that for any finite-dimensional $C^*$-algebra $A$ the space
\begin{equation}\label{eq:gen.pairs.sa}
  \left\{(a,b)\in A_{sa}^2\ :\ \text{$a$ and $b$ generate $A$ as a $C^*$-algebra}\right\}\subset A_{sa}^2
\end{equation}
of generating pairs of self-adjoint elements is dense. Somewhat more is true. Recall that an algebra $A$ (over a commutative ring, though we work mostly over fields here) is \emph{separable} \cite[\S 10.2, Definition]{pierce_assoc} if it is projective as a module over its \emph{enveloping algebra} \cite[\S 10.1, Definition]{pierce_assoc} $A^e:=A\otimes A^{op}$.

\begin{theorem}\label{th:gen.zar.den.sep}
  Let $A$ be a finite-dimensional algebra over a field $\Bbbk$.

  \begin{enumerate}[(1),wide]
  \item The subspace
    \begin{equation}\label{eq:gen.tups}
      \cat{Gen}_n(A)
      :=
      \left\{(a_i)_{i=1}^n\in A^n\ :\ \text{$a_i$ generate $A$ as a $\Bbbk$-algebra}\right\}\subset A^n
    \end{equation}
    constitutes the $\Bbbk$-rational points of a $\Bbbk$-defined open subscheme of $\Spec(A^n\otimes_{\Bbbk}\bK)$ for any algebraically-closed $\bK\ge \Bbbk$.    

  \item If $\Bbbk$ is infinite and $A$ is separable then $\cat{Gen}_2(A)$ is Zariski-dense.   
  \end{enumerate}
\end{theorem}
\begin{proof}
  \begin{enumerate}[(1),wide]
  \item Write
    \begin{equation*}
      \begin{aligned}
        \cat{Gen}_n(A)
        &=
          \bigcup_{m\in \bZ_{\ge 0}}\cat{Gen}^{\le m}_n(A)\\
        \cat{Gen}^{\le m}_n(A)
        &:=
          \left\{(a_i)\in A^n\ :\ A=\spn\left\{\text{degree-$(\le m)$ polynomials in the $a_i$}\right\}\right\}.
      \end{aligned}
    \end{equation*}
    Each $\cat{Gen}^{\le m}_n(A)$ is plainly $\Bbbk$-defined and open, and the union will stabilize (e.g. because by \cite[AG \S 3.4]{brl} the topological space $\Spec(A)$ is noetherian). 
    
  \item The infinitude of $\Bbbk$ ensures that
    \begin{equation*}
      A^n\le \overline{A}^n:=A^n\otimes_{\Bbbk}\overline{\Bbbk}
      \quad
      \left(\overline{\Bbbk}\ge \Bbbk\text{ an algebraic closure}\right)
    \end{equation*}
    is Zariski-dense; given openness, we may assume $\Bbbk$ algebraically closed to begin with ($\overline{A}$ also being separable \cite[\S 10.6, Proposition a]{pierce_assoc}). Moreover (again by openness), it suffices to prove \Cref{eq:gen.tups} non-empty.

    Decompose $A$ as a product $\prod_{i=1}^k M_{n_i}(\Bbbk)$ \cite[\S 3.5, Corollary b]{pierce_assoc}, embed it into $M_{n}(\Bbbk)$ block-diagonally for $n:=\sum n_i$, and consider the two operators
    \begin{equation*}
      a:=\mathrm{diag}(\lambda_j)_{j=1}^{n}
      \quad\text{and}\quad
      b:=\sum_{\substack{i\ne j\\e_{ij}\in A\subset M_n}}e_{ij}.
    \end{equation*}
    If the $\lambda_i$ are non-zero, distinct, with distinct $\lambda_{i}\lambda_j^{-1}$ for $1\le i\ne j\le n$ the two will generate the subalgebra $A\subset M_n$. 
  \end{enumerate}
\end{proof}

\begin{remark}\label{re:ss.nec}
  \Cref{th:gen.zar.den.sep} certainly does not hold for non-semisimple algebras: $\Bbbk[x_i,\ 1\le i\le n]/(x_i)^2$ cannot be generated by fewer than $n$ elements. 
\end{remark}

The following consequence of \Cref{th:gen.zar.den.sep} strengthens the aforementioned \cite[Lemma 7.2]{zbMATH07289424} as well as \cite[Lemma 3.10]{zbMATH07502493}, which it parallels. 

\begin{corollary}\label{cor:gen.cast.zar.den}
  For a finite-dimensional $C^*$-algebra $A$ the inclusion \Cref{eq:gen.pairs.sa} is Zariski-open and dense. 
\end{corollary}
\begin{proof}
  This is an immediate application of \Cref{th:gen.zar.den.sep} with $\Bbbk:=\bC$, once we observe that $A_{sa}\subset A$ is Zariski-dense.
\end{proof}

%\newpage

%%%%%%%%%%%%%%%%%%%%%%%%%%%%%%%%
%%%%%%%%%%%%%%%%%%%%%%%%%%%%%%%%

\addcontentsline{toc}{section}{References}
%\bibliography{bib}{}
%\bibliographystyle{plain}

\def\polhk#1{\setbox0=\hbox{#1}{\ooalign{\hidewidth
  \lower1.5ex\hbox{`}\hidewidth\crcr\unhbox0}}}

\Addresses

\end{document}